\documentclass{bcp92}
\usepackage[centertags]{amsmath}
\usepackage{amsfonts}
\usepackage{amssymb}
\def\R{\mathbb R}

\theoremstyle{plain}
\newtheorem{thm}{Theorem}[section]
\newtheorem{cor}[thm]{Corollary}
\newtheorem{lem}[thm]{Lemma}
\newtheorem{pro}[thm]{Proposition}
\theoremstyle{definition}
\newtheorem{defi}[thm]{Definition}
\newtheorem{exm}[thm]{Example}

\begin{document}
\nrart{1}
\setcounter{page}{15}
\mathclass{Primary 54E45, 47H08; Secondary 51F99, 05C05.}

\abbrevauthors{A. G. Aksoy and K. E. Kinneberg}
\abbrevtitle{Compact widths in metric trees}

\title{Compact widths in metric trees}

\author{Asuman G\"{u}ven Aksoy}
\address{Claremont McKenna College,
Department of Mathematics\\
Claremont, CA 91711, USA\\
E-mail: aaksoy@cmc.edu}

\author{Kyle Edward Kinneberg}
\address{University of California, Los Angeles,
Department of Mathematics\\
Los Angeles, CA 90095, USA\\
E-mail: kkinneberg@math.ucla.edu}

\maketitlebcp

\begin{abstract}
The definition of $n$-width of a bounded subset $A$ in a normed linear space
$X$ is  based on  the existence of $n$-dimensional subspaces. Although the
concept of  an $n$-dimensional subspace is not available for metric trees, in
this paper, using the properties of convex and compact subsets, we present a
notion of  $n$-widths for a  metric tree, called T$n$-widths. Later we
discuss properties of T$n$-widths, and show that the compact width is
attained. A relationship between the compact widths and T$n$-widths is also
obtained.
\end{abstract}

\section{Introduction}
The study of injective envelopes of metric spaces, also known as metric trees
($\mathbb{R}$-trees or T-theory) is motivated by many subdisciplines of
mathematics, biology/medicine and computer science. The relationship between
metric trees and biology and medicine stems from the construction of
phylogenetic trees \cite{ss}; and concepts of ``string matching" in computer
science are closely related with the structure of metric trees \cite{bcp}.

Unlike metric trees, in an ordinary tree  all the edges are assumed to have
the same length and therefore the metric is not often stressed. However, a
metric tree is a generalization of an ordinary tree that allows for different
edge lengths.  A metric tree is a metric space $(M,d)$ such that for every
$x,y$ in $M$ there is a unique arc between $x$ and $y$ isometric to an
interval in $\mathbb {R}$. For example, a connected graph without cycles is a
metric tree.  Metric trees also arise naturally in the study of group
isometries of hyperbolic spaces. For metric properties of trees we refer to
\cite{buneman}. Lastly, \cite{mmot} and \cite{mo} explore topological
characterization of metric trees and prove that for a separable metric space
$(M,d)$ the following are equivalent:
\begin{itemize}
\item $M$ admits an equivalent metric $\rho$ such that $(M,\rho)$ is a metric
tree.
\item  $M$ is locally arcwise connected and uniquely arcwise connected.
\end{itemize}
For an overview of geometry, topology, and group theory applications of
metric trees, consult Bestvina \cite{Best}. For a complete discussion of
these spaces and their relation to $CAT (\kappa)$ spaces we refer to
\cite{Brid}.

\begin{defi}\label{D:mseg}
    Let $x, y \in M$, where ($M$, $d$) is a metric space.  A \emph{geodesic
segment} from $x$ to $y$, is the image of an isometric embedding $\alpha :
[a,b]\rightarrow M$ such that $\alpha(a)=x$ and $\alpha(y)=b$. The geodesic
segment will be called  a \emph{metric segment} and denoted by $[x,y]$
throughout this paper.
\end{defi}

\begin{defi}\label{D:mt2}
    $(M,d)$, a metric space, is a \emph{metric tree} if and only if for all $x,y,z \in
    M$, the following holds:
    \begin{enumerate}
        \item there exists a unique metric segment from $x$ to $y$,
        and
        \item $[x,z] \cap [z,y] = \{z\} \Rightarrow [x,z] \cup [z,y] = [x,y]$.
    \end{enumerate}
\end{defi}

Note that $\mathbb{R}^n$ with the Euclidean metric satisfies the first
condition. It fails, however, to satisfy the second condition.  If the metric
$d$ is understood, we will denote $d(x,y)$ by $xy$. We also say that a point
$z$ is \textit{between} $x$ and $y$ if $xy = xz + zy$. We will often denote
this by $xzy$. It is not difficult to prove that in any metric space, the
elements of a metric segment from $x$ to $y$ are necessarily between $x$ and
$y$, and in a metric tree, the elements between $x$ and $y$ are the elements
in the unique metric segment from $x$ to $y$. Hence, if $M$ is a metric tree
and $x,y \in M$, then
$$
[x,y] = \{ z \in M : xy = xz + zy \}.
$$
The following is an example of a metric tree. For more examples see
\cite{AkTi}.

\begin{exm}[The radial metric]\label{E:radial}
    Define $d: \R^2 \times \R^2 \to \R_{\geq 0}$ by:
    \[
        d(x,y) =
            \begin{cases}
            \|x-y\| & \text{if $x = \lambda \, y$ for some $\lambda \in \R$,}\\
            \|x\| + \| y \| & \text{otherwise.}
            \end{cases}
    \]
    We can observe that the $d$ is in fact a metric and that $(\R^2,d)$ is a
metric tree.
\end{exm}

It is well known that any complete, simply connected Riemannian manifold
having non-positive curvature is a $CAT(0)$-space. Other examples include the
complex Hilbert ball with the hyperbolic metric (see \cite{Goebel}),
Euclidean buildings (see \cite{Brown}) and classical hyperbolic spaces.  If a
space is $CAT(\kappa)$ for some $\kappa < 0$ then it is automatically
$CAT(0)$-space. Although we will concentrate on metric trees, which is a
sub-class of $CAT(0)$-spaces, perhaps it is useful to mention the following:

\begin{pro}
If a metric space is $CAT(\kappa)$ space for all $\kappa$ then it is a metric
tree.
\end{pro}

 For the proof of the above proposition we refer to \cite{Brid}. Note that
if a Banach space is a $CAT(\kappa)$  space for some $\kappa$ then it is
necessarily a Hilbert space and $CAT(0)$. The property that distinguishes the
metric trees from the $CAT(0)$ spaces is the fact that metric trees are
hyperconvex metric spaces. Properties of hyperconvex spaces and their
relation to metric trees can be found in \cite{AkMa}, \cite{ap},
\cite{isbell} and \cite{kirk}.  We refer to \cite{Blum} for the properties of
metric segments and to \cite{AkBo} and \cite{AkKh} for the basic properties
of complete metric trees. In the following we list some of the properties of
metric trees which will be used throughout this paper.
\begin{enumerate}
\item (Transitivity of betweenness \cite{Blum}).  Let $M$ be a metric space
and let $a,b,c,d \in M$. If $abc$ and $acd$, then $abd$ and $bcd$.
\item  (Three point property \cite{AkBo}). Let $(M,d)$ be a metric tree and
$x,y,z \in M$.  There exists $w \in M$ such that $[x,z] \cap [y,z] = [w,z]$
and $[x,y] \cap [w,z] = \{w\}$.
\item  (Uniform convexity \cite{AkBo}). A metric tree $M$ is uniformly convex.
\end{enumerate}

\section{Kolmogorov $\boldsymbol n$-widths}
The following definition due to Kolmogorov \cite{Kolm}, gives a measure for
the ``thickness" or ``massivity" of a subset $A$ in a normed linear space
$X$. Kolmogorov $n$-widths have been widely used in approximation theory (see
\cite{Pink} and references therein). Recently $n$-width has been utilized as
a measure of efficiency in the task of data compression (see \cite{Dovd},
\cite{Posn}, \cite{Dono}). Furthermore, in \cite{AkTi},  entropy quantities,
other measures of compactness and $n$-affine Kolmogorov diameter were
studied in the context of metric trees.

\begin{defi}
Let $A$ be a subset of a normed linear space $X$, and let $\mathcal{X}_n$
denote the set of $n$-dimensional subspaces of $X$. We define the
\textit{Kolmogorov $n$-width} of $A$ in $X$ to be
$$
\delta_n(A,X) =
\inf_{X_n \in \mathcal{X}_n} \,\sup_{a \in A} \,\inf_{x \in X_n} \left\| x-a
\right\|
$$
\end{defi}

The left most infimum is taken over all $n$-dimensional subspaces $X_n$ of $X$.

Clearly $\delta_n(A,X)$ gives a measure the extent to which $A$  may be
approximated by $n$-dimensional subspaces of $X$. Indeed, it is easy to see
that if $A \subset X_n$ for some $X_n \in \mathcal{X}_n$, then $\delta_n(A,X)
= 0$. A subspace $X_n$ of $X$ of dimension at most $n$ for which
$$
\delta_n(A,X) = \sup_{a \in A} \,\inf_{x \in X_n} \left\| x-a \right\|,
$$
is called an \emph{optimal subspace} for $\delta_n(A,X)$. Generally it is
very difficult to calculate $\delta_n(A,X)$ and determine optimal subspaces
$X_n $ of $\delta_n(A,X)$, although a considerable effort has been devoted to
it. In many cases one is interested in determining asymptotic behavior of
$\delta_n(A,X)$ as $n \rightarrow \infty$. Aside from defining
$\delta_n(A,X)$, Kolmogorov  also computed this quantity for particular
spaces. The following is one of his examples:

\begin{exm}
Let $\tilde {W}_{2}^{(r)}$ denote the Sobolev space of $2 \pi$-periodic,
real-valued, $(r-1)$-times differentiable functions whose $(r-1)$th
derivative is absolutely continuous and whose $r$th derivative is in $L^2=
L^2[ 0, 2\pi]$. Set
$$
\tilde{B}_{2}^{(r)}= \{f:\,\, f\in \tilde
{W}_{2}^{r},\,\,\, \| f^{(r)}\| \leq 1 \,\,\}
$$
then
$$
\delta_0(\tilde{B}_{2}^{(r)}, L^2)= \infty \,\,\,\mbox{while} \,\, \delta_{2n-1}
(\tilde{B}_{2}^{(r)}, L^2)= \delta_{2n} (\tilde{B}_{2}^{(r)}, L^2)= n^{-r},
\quad n=1,2,\dots
$$
Furthermore, the optimal subspace for $\delta_{2n}
(\tilde{B}_{2}^{(r)}, L^2)$ is the set of trigonometric polynomials of degree
less than or equal to $n-1$; namely,
$$
T_{n-1} = \mbox{span} \{1,\sin x, \cos x, \dots , \sin(n-1)x, \cos(n-1)x \}.
$$
\end{exm}

It is natural to ask whether or not we can alter the traditional definition
so that $n$-widths can be defined in metric trees. The obvious replacement
for $||x-y||$ is $d(x,y)$. The more difficult alteration, however, is
defining ``dimension'' of a set in a metric space.  In the following, we
attempt to remedy this problem for metric trees.

\subsection {$\boldsymbol n$-widths of metric trees via convexity}
We call a subset $A$ of a metric tree $(M,d)$ \textit{convex} if for any $x,y
\in A$, the metric segment $[x,y]$ is in $A$.  By definition, every metric
tree is convex. The converse is also true: it is easy to see that, if
$(M,d)$ is a metric tree and $A \subset M$ is a convex subset of $M$, then
$A$ is a metric tree. For $B \subset M$, the convex hull of $B$, denoted
by $\text{conv}(B)$, is the smallest convex set that contains $B$, where the
order is set inclusion.

\begin{defi}
Let $(M,d)$ be a metric tree, and let $A \subseteq M$. We say that $A$ is
T\textit{$n$-dimensional\/} if and only if there exist $n$ points $x_1, \dots, x_n \in
M$ such that
$$
A = \text{conv}(x_1, \dots, x_{n})
$$
and there do not exist
$i \neq j \neq k$ such that $x_i x_j x_k$. Also, we say that $A$ is
T*\textit{$n$-dimensional\/} if $A$ contains a T$n$-dimensional subset but
does not contain any T$k$-dimensional subsets for all $k > n$.
\end{defi}

Note that the restriction $i \neq j \neq k$ tells us that the points
$x_1,\dots,x_n$ are all distinct.

\begin{lem} \label{convex-hull}
If $(M,d)$ is a metric tree and $A$ is a subset of $M$, then
$$
\text{\rm conv}(A) = \{ z \in M : xzy \text{ for some } x,y \in A \}.
$$
\end{lem}

\begin{proof}
First, we observe that $C = \{ z \in M : xzy \text{ for some } x,y \in A \}$
is a convex set. Indeed, if $a,b \in C$, then $xay$ and $ubw$ for some
$x,y,u,w \in A$. By definition of $C$, the segment $[y,u]$ is in $C$, as are
the segments $[a,y]$ and $[u,b]$. Since $[a,b] \subseteq [a,y] \cup [y,u]
\cup [u,b]$, we know that $[a,b] \subseteq C$. Thus, $C$ is convex. Also, $C$
contains $A$, so by definition of convex hull, we have that $\text{conv}(A)
\subseteq C$.

Now, let $z \in C$. Then there exist some $x,y \in A$ such that $xzy$. Hence,
$z \in [x,y]$ and $[x,y] \subseteq \text{conv}(A)$, so $z \in
\text{conv}(A)$. Thus, we indeed have $C \subseteq \text{conv}(A)$ as
desired.
\end{proof}

An important concept regarding metric trees is that of ``final points''. We
have the following definition and subsequent theorem.

\begin{defi}
Let $(M,d)$ be a metric tree, and let $A \subseteq M$. We call
$$
F_A = \left\{ f \in A : f \notin (x,y) \text{ for all } x,y \in A \right\}
$$
the set of \textit{final points} of $A$. Here, $(x,y) = [x,y] \backslash \{
x,y \}$.
\end{defi}

\begin{thm}[\cite{AkBo}] \label{compact-tree}
A metric tree $(M,d)$ is compact if and only if
$$
M = \bigcup_{f \in
F_M}{[a,f]} \text{ for all } a \in M, \text{ and } \overline{F}_M \text{ is
compact.}
$$
\end{thm}

We now characterize T$n$-dimensional subsets of a metric tree, and establish
several facts about such subsets. Here, $(M,d)$ will be a metric tree and
$\mathcal{X}_n$ will denote the set of all T$n$-dimensional subsets of $M$.

\begin{thm} \label{dim-char}
Let $A$ be a subset of $M$. Then $A$ is {\rm T}$n$-dimensional if and only if
$A$ is a compact metric tree with $F_A = \{ x_1, \dots, x_{n} \}$ for some
$x_1, \dots, x_{n} \in M$.
\end{thm}

\begin{proof}
Let $A$ be T$n$-dimensional. Then there exist $x_1, \dots, x_{n} \in M$ such
that $A = \text{conv}(x_1, \dots, x_{n})$ and there do not exist $i \neq j
\neq k$ such that $x_i x_j x_k$. By Lemma \ref{convex-hull}, $A = \{ z \in M
: x_i z x_j \text{ for some } x_i,x_j \}$. Note that $A$ is a metric tree
because it is convex. We now show that $A$ is compact.

To do this, we first show that for any $a \in A$, $ A =
\bigcup_{i=1}^{n}{[a,x_i]}$. Let $a \in A$ be fixed. If $z \in
\bigcup_{i=1}^{n}{[a,x_i]}$, then $z \in [a,x_i]$ for some $i$. Since $A$ is
convex, $[a,x_i] \subseteq A$, so $z \in A$. Now, let $z \in A$. Then there
exist $i,j$ such that $z \in [x_i,x_j]$. We know by the three point property
that there is some $w \in A$ such that $[a,x_i] \cap [a,x_j] = [a,w]$ and
$[x_i,x_j] \cap [a,w] = \{ w \}$. Note that $w \in [x_i,x_j]$, so $[x_i,w]
\cup [x_j,w] = [x_i,x_j]$. Since $w \in [a,x_i]$ and $w \in [a,x_j]$, we have
$[x_i,w] \subseteq [a,x_i]$ and $[x_j,w] \subseteq [a,x_j]$. Thus,
$$
z \in [x_i,x_j] = [x_i,w] \cup [x_j,w] \subseteq [a,x_i] \cup [a,x_j]
\subseteq \bigcup_{i=1}^{n}{[a,x_i]}.
$$
Hence, $A = \bigcup_{i=1}^{n}{[a,x_i]}$, as desired.

We now show that $F_A = \{ x_1, \dots, x_{n} \}$. Since $A = \{ z \in M : x_i
z x_j \text{ for some } x_i,x_j \}$, we see that $A = \{ z \in M : z \in
(x_i,x_j) \text{ for some } x_i,x_j \} \cup \{ x_1, \dots, x_{n} \}$. Thus,
the only possible final points of $A$ are $x_1, \dots, x_{n}$. If, for some
$j$, $x_j$ is not a final point, then there must exist some $y,z \in A$ such
that $x_j \in (y,z)$. Since $ A = \bigcup_{i=1}^{n}{[a,x_i]} $ for any $a \in
A$, we know that there exist some $x_i$ and $x_k$ such that $y \in [z,x_i]$
and $z \in [y,x_k]$. Therefore, we have $y,z \in [x_i,x_k]$, so $[y,z]
\subseteq [x_i,x_k]$. But this implies that $x_j \in (x_i,x_k)$, contrary to
our assumption about the set $\{ x_1, \dots, x_{n} \}$. Hence, $x_j \in F_A$
for all $1 \leq j \leq n$, so $F_A = \{ x_1, \dots, x_{n} \}$, as claimed. In
particular, this implies that $\overline{F}_A = F_A$ is compact. By Theorem
\ref{compact-tree} then, $A$ is compact, so $A$ is a compact tree with $F_A =
\{ x_1, \dots, x_{n} \}$.

Now let $A$ be a compact metric tree with $F_A = \{ x_1, \dots, x_{n} \}$ for
some $x_1, \dots, x_{n} \in M$. Note first that there do not exist $i \neq j
\neq k$ such that $x_i x_j x_k$. We want to show that $A = \text{conv}(x_1,
\dots, x_{n})$. By definition of final points, we actually have $x_1, \dots,
x_{n} \in A$, and since $A$ is a metric tree, it is convex. Thus,
$\text{conv}(x_1, \dots, x_{n}) \subseteq A$. Now, let $z \in A$. Since $A$
is compact, Theorem \ref{compact-tree} tells us that for all $a \in A$, $ A =
\bigcup_{i=1}^{n}{[a,x_i]}$. Therefore, we have $ z \in A =
\bigcup_{i=1}^{n}{[x_1,x_i]}$, so there is some $i$ for which $z \in
[x_1,x_i]$. Thus, $z \in \text{conv}(x_1, \dots, x_{n})$, which implies that
$A = \text{conv}(x_1, \dots, x_{n})$. Hence, $A$ is T$n$-dimensional, as
desired.
\end{proof}

\begin{lem} \label{lower-dim}
Every {\rm T}$n$-dimensional subset $X_n$ of $M$ contains a {\rm
T}$m$-dimensional subset for each $1 \leq m \leq n$.
\end{lem}

\begin{proof}
Let $X_n = \text{conv}(x_1, \dots, x_{n})$ such that there do not exist $i
\neq j \neq k$ where $x_i x_j x_k$. If $1 \leq m \leq n$, let $X_m =
\text{conv}(x_1, \dots, x_{m})$, so that $X_m$ is T$m$-dimensional and $X_m
\subset X_n$.
\end{proof}

\begin{lem} \label{nolarger}
If $X_m \subseteq X_n$ for some $X_m \in \mathcal{X}_m$ and $X_n \in
\mathcal{X}_n$, then $m \leq n$.
\end{lem}

\begin{proof}
Let $X_m = \text{conv}(x_1, \dots, x_{m})$ and $X_n = \text{conv}(y_1, \dots,
y_{n})$. By Theorem \ref{dim-char}, this implies that $F_{X_m} = \{ x_1,
\dots, x_{m} \}$ and $F_{X_n} = \{ y_1, \dots, y_{n} \}$. We can assume that
$X_m \neq X_n$, so there is some $j$ for which $x_i \neq y_j$ for all $i \in
\{ 1, \dots, m \}$. Without loss of generality, let this $j$ be 1. By the
compactness of $X_n$, Theorem \ref{compact-tree} tells us that for any $a \in
X_n$,
$$
X_n = \bigcup_{i=1}^{n}{[a,y_i]}.
$$
Therefore, we have
$$
X_m \subseteq \bigcup_{i=1}^{n}{[y_1,y_i]} ,
$$
so for any $x_j \in F_{X_m}$, we see that $x_j \in [y_1,y_k]$ for some $k \in
\{ 1, \dots, n \}$. Now define a function $f: F_{X_m} \rightarrow F_{X_n}$ by
the following:
$$
f(x_i) = \begin{cases}
y_k & \text{ if } x_i \in [y_1,y_k] \text{ and no other element of } F_{X_m}
\text{ is in } [x_i,y_k] \\
y_1 & \text{ otherwise}
\end{cases}
$$
and if the first condition holds for more than one $k$, choose the smallest
of such $k$.

Clearly, $f$ is well-defined for all $x_i \in F_{X_m}$. We want to show that
$f$ is an injection. Suppose for a contradiction that $f(x_i) = y_k = f(x_j)$
for some $i \neq j$ in $\{ 1, \dots, m \}$ and $k \in \{ 2, \dots, n \}$.
Then by definition of $f$, we have $x_i,x_j \in [y_1,y_k]$, which implies
that either $x_i \in [x_j,y_k]$ or $x_j \in [x_i,y_k]$. If the former, then
we contradict the fact that no element of $F_{X_m}$, other than $x_j$, is in
$[x_j,y_k]$; and if the latter, we contradict the fact that no element of
$F_{X_m}$, other than $x_i$, is in $[x_i,y_k]$.

So now suppose that $f(x_i) = y_1 = f(x_j)$ for some $i \neq j$ in $\{ 1,
\dots, m \}$. Note that $x_i \neq y_1 \neq x_j$, so $x_i$ and $x_j$ are
mapped to $y_1$ by the ``otherwise'' condition, not by the first condition.
We now claim that if some $x_k$ is mapped to $y_1$ by the ``otherwise''
condition, then the segment $(x_k,y_1]$ does not contain any elements of
$X_m$. Indeed, since $x_k \in [y_1,y_k]$ for some $k \neq 1$, and since $x_k$
is not mapped to $y_k$, there must be some $x_l \in (x_k,y_k]$. Now, if there
was some $w \in X_m$ in the segment $(x_k,y_1]$ then we would have $x_k \in
(x_l,w)$, which contradicts our assumption that $x_k$ is a final point of
$X_m$. Thus, $(x_k,y_1]$ does not contain any elements of $X_m$.

We therefore know that $(x_i,y_1]$ and $(x_j,y_1]$ contain no elements of
$X_m$. But then by the three point property, there is a $w \in M$ such that
$[x_i,y_1] \cap [x_j,y_1] = [w,y_1]$ and $[x_i,x_j] \cap [w,y_1] = \{ w \}$.
Since $w \in [x_i,x_j]$, $w \in X_m$. If $x_i \neq w$, then $w \in
(x_i,y_1]$, and if $x_j \neq w$, then $w \in (x_j,y_1]$. Both possibilities
contradict the fact that $(x_i,y_1]$ and $(x_j,y_1]$ contain no elements of
$X_m$. Hence, $x_i = w = x_j$, another contradiction. Therefore, no two
distinct elements of $F_{X_m}$ can map to $y_1$. We can then conclude that
$f$ is an injection.

Since $F_{X_m}$ and $F_{X_n}$ are finite sets, the injectivity of $f$ implies
that $| F_{X_m} | \leq | F_{X_n} |$, so $m \leq n$. Notice also that the
function $f$ is a bijection if and only if $m = n$.
\end{proof}

Now that we have established some facts about T$n$-dimensional subsets of a
metric tree, we can give the following definition for the T$n$-width.

\begin{defi}
Let $A$ be a subset of a metric tree $(M,d)$, and let $\mathcal{X}_n$ denote
the set of T$n$-dimensional subsets of $M$. We define the T\textit{$n$-width}
of $A$ to be
$$
\delta_n^T(A,M) = \inf_{X \in \mathcal{X}_n} \sup_{a \in A}
\inf_{x \in X} d(a,x).
$$
If $M$ is T*$n$-dimensional (i.e., $M$ does not
contain any T$k$-dimensional subsets for $k > n$ but does contain a
T$n$-dimensional subset), then by convention we say that $\delta_k^T(A,M) =
\delta_n^T(A,M)$.
\end{defi}

First observe that if $A$ is unbounded, then $\delta_n^T(A,M) = \infty$ for
each $n \in \mathbb{N}$. Indeed, this follows directly from the fact that
every T$n$-dimensional set is bounded. Conversely, it is easy to see that if
$A$ is bounded, then $\delta_n^T(A,M) < \infty$ for each $n$. Therefore, we
really will be interested only in the T$n$-widths of bounded sets.

\begin{exm}
Let $M = \mathbb{R}^k$ endowed with the radial metric. If $B_r$ denotes the
(open or closed) ball of radius $r$ in $\mathbb{R}^k$, then
$\delta_n^T(B_r,M) = r$ for all $n \in \mathbb{N}$.
\end{exm}

To see this, let $n \in \mathbb{N}$. Choose an $X_n \in \mathcal{X}_n$ such
that the origin is in $X_n$. Since $d(a,0) \leq r$ for any $a \in B_r$, we
have $\inf_{x \in X_n} d(a,x) \leq r$. Thus, $\sup_{a \in B_r} \inf_{x \in
X_n} d(a,x) \leq r$, so $\delta_n^T(B_r,M) \leq r$.

Now we must show that for each $X_n \in \mathcal{X}_n$, $\sup_{a \in B_r}
\inf_{x \in X_n} d(a,x) \geq r$. If $X_n \in \mathcal{X}_n$, then there exist
points $x_1, \dots, x_{n}$ such that $X_n = \text{conv}(x_1, \dots, x_{n})$.
Now, choose a ray $v$ beginning at the origin such that $v$ contains none of
the $x_i$'s. This implies that $v$ contains no points in $X_n$, with the
possible exception of the origin. Now, for each $\varepsilon > 0$, we can
find a point $p_{\varepsilon}$ in $v \cap B_r$ for which
$d(p_{\varepsilon},0) > r - \varepsilon$. Then, if $x \in X_n$, we know that
$d(p_{\varepsilon}, x) = d(p_{\varepsilon},0) + d(x,0)$ since $x$ and
$p_{\varepsilon}$ do not lie on the same ray. Thus, $d(p_{\varepsilon}, x) >
r - \varepsilon$. Hence, for each $\varepsilon > 0$, $\inf_{x \in X_n}
d(p_{\varepsilon}, x) > r - \varepsilon$, so $\sup_{a \in B_r} \inf_{x \in
X_n} d(a,x) \geq r$. Therefore, $\delta_n^T(B_r,M) \geq r$.

\def\labelenumi{{\rm\arabic{enumi}.}}

In the following we first give basic properties of $\delta_n^T(A,M) < \infty$.

\begin{pro} \label{noninc}
Let $A \subseteq B$ be subsets of $M$. Then
\begin{enumerate}
\item For any $n \in \mathbb{N}$, $\delta_n^T(A,M) \leq \delta_n^T(B,M)$.
\item The sequence $\{ \delta_n^T(A,M) \}_{n \in \mathbb{N}}$ is
non-increasing.
\end{enumerate}
\end{pro}

\begin{proof}
$1.$ Let $n \in \mathbb{N}$ such that $M$ has at least one T$n$-dimensional
subset. Let $X \in \mathcal{X}_n$. Since $A \subseteq B$,
$$
\sup_{a \in A}\, \inf_{x \in X} d(a,x) \leq \sup_{b \in B}\, \inf_{x \in X}
d(b,x).
$$
This holds for any $X \in \mathcal{X}_n$, so we have
$$
\inf_{X \in \mathcal{X}_n}\, \sup_{a \in A} \,\inf_{x \in X} d(a,x) \leq
\inf_{X \in \mathcal{X}_n}\, \sup_{b \in B} \,\inf_{x \in X} d(b,x).
$$
Hence, $\delta_n^T(A,M) \leq \delta_n^T(B,M)$.

If $n \in \mathbb{N}$ such that $M$ has no T$n$-dimensional subsets, then
there is a $k < n$ such that $M$ is T*$k$-dimensional. By definition, $M$ has
at least one T$k$-dimensional subset. Thus, by what we just found,
$\delta_k^T(A,M) \leq \delta_k^T(B,M)$. By convention, $\delta_k^T(A,M) =
\delta_n^T(A,M)$ and $\delta_k^T(B,M) = \delta_n^T(B,M)$, so $\delta_n^T(A,M)
\leq \delta_n^T(B,M)$. Thus, for any $n \in \mathbb{N}$, $\delta_n^T(A,M)
\leq \delta_n^T(B,M)$.

\vsks

$2.$ Suppose that the sequence is increasing somewhere. Then $A$ must be
bounded (since otherwise each T$n$-width is $\infty$) and there is an $n \in
\mathbb{N}$ such that $\delta_n^T(A,M) <  \delta_{n+1}^T(A,M)$. Thus, there
is some $X_n \in \mathcal{X}_n$ (say $X_n = \text{conv}(x_1, \dots, x_{n})$)
such that
$$
\sup_{a \in A} \,\inf_{x \in X_n} d(a,x) <
\sup_{a \in A} \,\inf_{x \in X_{n+1}} d(a,x)
$$
for all $X_{n+1} \in \mathcal{X}_{n+1}$. We claim that for any finite set $\{
y_1, \dots, y_m \}$ in $M$, the set $Y_m = \text{conv}(x_1, \dots, x_{n},
y_1, \dots, y_m)$ is T$n$-dimensional.

Suppose that $Y_m$ is not T$n$-dimensional. Since $X_n \subseteq Y_m$, we
know by Lemma \ref{nolarger} that $Y_m$ is no less than T$n$-dimensional.
Thus, $Y_m$ is T$(n+k)$-dimensional for some $k \geq 1$. By removing some of
the $y_i$'s if $k > 1$, we can produce a set $Y_p = \text{conv}(x_1, \dots,
x_{n}, y_1, \dots, y_p)$ such that $Y_p$ is T$(n+1)$-dimensional. Now, since
$X_n \subseteq Y_p$, for any $a \in A$, we have
$$
\inf_{x \in X_n} d(a,x) \geq \inf_{x \in Y_p} d(a,x).
$$
Therefore,
$$
\sup_{a \in A} \,\inf_{x \in X_n} d(a,x) \geq
\sup_{a \in A} \,\inf_{x \in Y_p} d(a,x).
$$
But since $Y_p$ is T$(n+1)$-dimensional, this contradicts the fact that
$$
\sup_{a \in A} \,\inf_{x \in X_n} d(a,x) <
\sup_{a \in A} \,\inf_{x \in X_{n+1}} d(a,x)
$$
for all $X_{n+1} \in \mathcal{X}_{n+1}$. Hence, $Y_m$ must be T$n$-dimensional.

Now, suppose that there exists some $X_{n+1} \in \mathcal{X}_{n+1}$. Then
$X_{n+1} = \text{conv}(y_1, \dots, y_{n+1})$ for some $y_1, \dots, y_{n+1}
\in M$. By what we just established, the set
$$
Y_m = \text{conv}(x_1, \dots, x_{n}, y_1, \dots, y_{n+1})
$$
is T$n$-dimensional. But then $X_{n+1} \subseteq Y_m$, so we have a
T$(n+1)$-dimensional set within a T$n$-dimensional set, contrary to Lemma
\ref{nolarger}. We can therefore conclude that $M$ does not contain any
T$(n+1)$-dimensional sets, so by Lemma \ref{lower-dim}, $M$ does not contain
any T$k$-dimensional sets for $k > n$. Hence, by definition, $M$ is
T*$n$-dimensional, so by convention, $\delta_n^T(A,M) = \delta_{n+1}^T(A,M)$,
a contradiction. Thus, the sequence is non-increasing.
\end{proof}

\subsection{Compact widths}
A concept that is related to the $n$-width is the \textit{compact width}.
Given a metric space $(M,d)$, let $\mathcal{X}$ denote the set of compact
subsets of $M$. If $A$ is a subset of $M$, we define the compact width of $A$
to be
$$
a(A,M) = \inf_{X \in \mathcal{X}} \,\sup_{a \in A} \,\inf_{x \in X} d(a,x).
$$
Observe that like T$n$-widths, $a(A,M) = \infty$ if and only if $A$ is
unbounded. Indeed, this follows easily from the fact that every compact set
in a metric space is bounded. We also have the following lemma.

\begin{lem} \label{geq}
If $(M,d)$ is a metric tree with subset $A$, then  $\delta_n^T(A,M) \geq
a(A,M)$ for all $n \in \mathbb{N}$.
\end{lem}

\begin{proof}
This is a direct consequence of Theorem \ref{dim-char}. Since each
T$n$-dimensional subset of~$M$ is compact, $\mathcal{X}_n \subseteq
\mathcal{X}$ for all $n \in \mathbb{N}$. Hence, for each $n$,
$$
\delta_n^T(A,M) = \inf_{X \in \mathcal{X}_n} \,\sup_{a \in A} \,\inf_{x \in
X} d(a,x) \geq
\inf_{X \in \mathcal{X}} \,\sup_{a \in A} \,\inf_{x \in X} d(a,x) = a(A,M),
$$
as desired.
\end{proof}

\begin{defi}
We say that a metric space $(X,d)$ \textit{has the property} $P_1$ if for every
$\varepsilon > 0$ and $r > 0$, there is a $\delta > 0$ such that for each
$x,y \in X$, there is a $z \in B(x,\varepsilon)$ for which $B(x,r+\delta)
\cap B(y,r+\theta) \subseteq B(z,r+\theta)$ if $0 < \theta < \delta$.
\end{defi}

Property $P_1$  was studied by several authors, for example see \cite{Amir},
\cite{March}.  The following theorem establishes a relationship between the
property $P_1$ and compact widths.

\begin{thm} [\cite{kamal}] \label{P1}
Let $X$ be a Banach space. If $X$ has the property $P_1$, then for each
bounded subset $A$ of $X$, the compact width $a(A,X)$ is attained.
\end{thm}

\begin{thm} \label{tree-P1}
Every metric tree has the property $P_1$.
\end{thm}

\begin{proof}
Let $(M,d)$ be a metric tree, and let $\varepsilon > 0$ and $r > 0$ be given.
Choose any $0 < \delta < \varepsilon$. We claim that such
$\delta$ works regardless of~$r$.

Let $x,y \in M$, and first suppose that $xy \geq \delta$. Choose $z \in
[x,y]$ such that $xz = \delta$, and since $\delta < \varepsilon$, we have $z
\in B(x,\varepsilon)$. We claim that with this choice of $\delta$ and $z$, we
have $B(x,r+\delta) \cap B(y,r+\theta) \subseteq B(z,r+\theta)$ for $0 <
\theta < \delta$.

Suppose $w \in B(x,r+\delta) \cap B(y,r+\theta)$. By the three point
property, there exists a $u \in M$ such that $[x,w] \cap [w,y] = [w,u]$ and
$[x,y] \cap [w,u] = \{ u \}$. Since $u \in [x,y]$ and $z \in [x,y]$, we know
that either $z \in [x,u]$ or $z \in [u,y]$. If $z \in [x,u]$, then we have
the following:
$$
\begin{aligned}
zw &= xw - xz &&\text{ since } z \in [x,u] \text{ and } u \in [x,w] \text{
implies that } z \in [x,w] \\
&< r + \delta - \delta &&\text{ since } w \in B(x,r+\delta) \text{ and }
xz = \delta \\
&\leq r + \theta &&\text{ since } \theta > 0.
\end{aligned}
$$
If $z \in [u,y]$, then we have the following:
$$
\begin{aligned}
zw &= wy - zy &&\text{ since } z \in [y,u] \text{ and } u \in [y,w] \text{
implies that } z \in [y,w] \\
&< r + \theta - zy &&\text{ since } w \in B(y,r+\theta) \\
&\leq r + \theta &&\text{ since } zy \geq 0.
\end{aligned}
$$
Therefore, in either case, $w \in B(z,r+\theta)$, so $B(x,r+\delta) \cap
B(y,r+\theta) \subseteq B(z,r+\theta)$.

Now suppose that $xy < \delta$. In this case, choose $z = y$. Since $\delta <
\varepsilon$, we have $z \in B(x,\varepsilon)$, and since $B(y,r+\theta) =
B(z,r+\theta)$, we have $B(x,r+\delta) \cap B(y,r+\theta) \subseteq
B(z,r+\theta)$ for $0 < \theta < \delta$.

Thus, by choosing $0<\delta < \varepsilon $, for any $x,y \in M$
there is a $z \in M$ for which $B(x,r+\delta) \cap B(y,r+\theta) \subseteq
B(z,r+\theta)$ for $0 < \theta < \delta$.
\end{proof}

\begin{cor} \label{tree-attained}
For any bounded subset $A$ of a complete metric tree $(M,d)$, the compact width
$a(A,M)$ is attained.
\end{cor}

\begin{proof}
In \cite{kamal}, Theorem \ref{P1} above is proved for Banach spaces. However,
the proof uses none of the linear structure of a Banach space; it applies
equally well to complete metric spaces.
\end{proof}

\begin{thm} \label{lim-delta}
For any subset $A$ of a metric tree $(M,d)$,
$$
\lim_{n \rightarrow \infty}{\delta_n^T(A,M)} = a(A,M).
$$
Here we take the convention
that if $M$ is {\rm T*}$k$-dimensional then $\delta_n^T(A,M) =
\delta_k^T(A,M)$ for $n > k$.
\end{thm}

\proof
First, observe that if $A$ is unbounded, then $\delta_n^T(A,M) = \infty =
a(A,M)$ for all $n$, so the result holds trivially. Therefore, suppose that
$A$ is bounded. Let $\varepsilon > 0$ be given, and let $X \in \mathcal{X}$
such that
$$
\sup_{a \in A} \,\inf_{x \in X} d(a,x) \leq a(A,M) + \frac{\varepsilon}{3}\,.
$$
We now wish to approximate $X$ by a T$n$-dimensional set. By the compactness
of $X$, there exists a finite set of points $x_1, \dots, x_{m}$ in $X$ such
that
$$
X \subseteq \bigcup_{i=1}^{m}B\Bigl(x_i,\frac{\varepsilon}{3}\Bigr).
$$
Let $X_n = \text{conv}(x_1, \dots, x_{m})$, so $X_n$ is a T$n$-dimensional
set for some $n \leq m$.

Let $a \in A$. Then there exists a $y \in X$ such that $d(a,y) \leq \inf_{x
\in X} d(a,x) + \frac{\varepsilon}{3}$. Also, there exists an $x_i$ such that
$d(x_i,y) \leq \frac{\varepsilon}{3}$. We therefore have
$$
d(a,x_i) \leq d(a,y) + d(x_i,y) \leq \inf_{x \in X} d(a,x) +
\frac{\varepsilon}{3} + \frac{\varepsilon}{3} = \inf_{x \in X} d(a,x) +
\frac{2 \varepsilon}{3}\,.
$$
Hence, we know that
$$
\inf_{x \in X_n} d(a,x) \leq \inf_{x \in X} d(a,x) + \frac{2 \varepsilon}{3}
$$
for any $a \in A$. This implies that
$$
\sup_{a \in A} \,\inf_{x \in X_n} d(a,x) \leq
\sup_{a \in A} \,\inf_{x \in X} d(a,x) + \frac{2 \varepsilon}{3} \leq a(A,M)
+ \frac{2 \varepsilon}{3} + \frac{\varepsilon}{3}\, ,
$$
and as a result,
$$
\delta_n^T(A,M) = \inf_{X_n \in \mathcal{X}_n} \,\sup_{a \in A} \,\inf_{x \in
X_n} d(a,x) \leq
\sup_{a \in A} \,\inf_{x \in X_n} d(a,x) \leq a(A,M) + \varepsilon.
$$
Therefore, for every $\varepsilon > 0$, there exists an $n \in \mathbb{N}$
such that $\delta_n^T(A,M) \leq a(A,M) + \varepsilon$.

By Proposition \ref{noninc}, the sequence $\{ \delta_n^T(A,M) \}_{n \in
\mathbb{N}}$ is non-increasing. If $\varepsilon > 0$ is given, choose $N \in
\mathbb{N}$ such that $\delta_N^T(A,M) \leq a(A,M) + \varepsilon$. We then
know that $\delta_n^T(A,M) \leq a(A,M) + \varepsilon$ for any $n \geq N$, so
$$
\delta_n^T(A,M) - a(A,M) \leq \varepsilon.
$$
Also by Lemma \ref{geq}, we know that for all $n \in \mathbb{N}$,
$\delta_n^T(A,M) \geq a(A,M)$. Hence,
$$
| \delta_n^T(A,M) - a(A,M) | = \delta_n^T(A,M) - a(A,M) \leq \varepsilon.
$$
Therefore,
$$
\lim_{n \rightarrow \infty}{\delta_n^T(A,M)} = a(A,M).\qquad\sq
$$

\makeatletter
\immediate\write\@mainaux{\string\@input{bcp\volno.aux}}%
\makeatother

\end{document}